\documentclass[11pt,reqno]{amsart}
\usepackage{amscd,amssymb,amsmath,amsthm}
\usepackage[arrow,matrix]{xy}
\usepackage{graphicx}
\usepackage{cite}
\usepackage{geometry}
\newtheorem{thm}[subsection]{Theorem}
\newtheorem{lemma}[subsection]{Lemma}
\newtheorem{pro}[subsection]{Proposition}

\numberwithin{equation}{section} 

\newcommand{\bea}{\begin{eqnarray}}
\newcommand{\eea}{\end{eqnarray}}

\newcommand{\R}{\mathbb{R}}

\begin{document}
\title[FIXED POINTS OF LYAPUNOV INTEGRAL OPERATORS]{FIXED POINTS OF LYAPUNOV INTEGRAL OPERATORS AND GIBBS MEASURES}

\author{F. H. Haydarov}

\address{F.\ H.\ Haydarov\\ National University of Uzbekistan,
Tashkent, Uzbekistan.} \email {haydarov\_imc@mail.ru}

\begin{abstract}In this paper we shall consider the connections between Lyapunov integral operators and
 Gibbs measures for four competing interactions of models with uncountable (i.e. $[0,1]$)
set of spin values on a Cayley tree. And we shall prove the
existence of fixed points of Lyapunov integral operators and give
a condition of uniqueness of fixed points.

\end{abstract}
\maketitle



\section{Preliminaries}

 A Cayley tree $\Gamma^k=(V,L)$ of order $k\in
\mathbb{N}$ is an infinite homogeneous tree, i.e., a graph without
cycles, with exactly $k+1$ edges incident to each vertices. Here
$V$ is the set of vertices and $L$ that of edges (arcs). Two
vertices $x$ and $y$ are called nearest neighbors if there exists
an edge $l\in L$ connecting them. We will use the notation
$l=\langle x,y\rangle$. The distance $d(x,y), x,y \in V$ on the
Cayley tree is defined by the formula
$$d(x,y)=\min\{d |\ x=x_{0},x_{1},...,x_{d-1},x_{d}=y\in V \ \emph{such that the pairs}$$
$$<x_{0},x_{1}>,...,<x_{d-1},x_{d}> \emph{are neighboring vertices}\}.$$
Let $x^{0}\in V$ be a fixed and we set
$$W_{n}=\{x\in V\ |\ d(x,x^{0})=n\}, \,\,\,\,\ V_{n}=\{x\in V\ |\ d(x,x^{0})\leq n\},$$
$$L_{n}=\{l=<x,y>\in L\ |\ x,y \in V_{n}\},$$
 The set of the direct successors of $x$ is denoted by $S(x),$
 i.e.
 $$S(x)=\{y\in W_{n+1}|\ d(x,y)=1\}, \ x\in W_{n}.$$
 We observe that for any vertex $x\neq x^{0},\ x$ has $k$ direct
 successors and $x^{0}$ has $k+1$. The vertices $x$ and $y$ are called second neighbor which is denoted by
 $>x,y<,$ if there exist a vertex $z\in V$ such that
 $x$, $z$ and $y$, $z$ are nearest neighbors. We will consider only second neighbors $> x, y <,$ for which there
exist $n$ such that $x, y \in W_n$. Three vertices $x,\ y$ and $z$
are called a triple of neighbors and they are denoted by $< x, y,
z>,$ if $< x, y >,\ < y, z >$ are nearest neighbors and $x,\ z \in
W_n,\ y \in W_{n-1}$, for some $n \in \mathbb{N}$.

Now we consider models with four competing interactions where the
spin takes values in the set
 $[0,1]$. For some set $A\subset V$ an arbitrary function $\sigma_A:A\to
[0,1]$ is called a configuration and the set of all configurations
on $A$ we denote by $\Omega_A=[0,1]^A$. Let $\sigma(\cdot)$ belong
to $\Omega_{V}=\Omega$ and $\xi_{1}:(t,u,v)\in[0,1]^{3}\to
\xi_{1}(t,u,v)\in R$,  $\xi_{i}: (u,v)\in [0,1]^2\to
\xi_{i}(u,v)\in R, \ i\in \{2,3\}$ are given bounded, measurable
functions. Then we consider the model with four competing
interactions on the Cayley tree which is defined by following
Hamiltonian
$$H(\sigma)=-J_{3}\sum_{<x,y,z>}\xi_{1}\left(\sigma(x),\sigma(y),\sigma(z)\right)
-J\sum_{>x,y<}\xi_{2}\left(\sigma(x),\sigma(z)\right)$$
\begin{equation}\label{e1}
-J_{1}\sum_{<x,y>}\xi_{3}\left(\sigma(x),\sigma(y)\right)-\alpha\sum_{x\in
V}\sigma(x),
\end{equation}
 where the sum in the first term ranges all triples of
neighbors, the second sum ranges all second neighbors, the third
sum ranges all nearest neighbors and  $J, J_{1}, J_{3},\alpha\in
R\setminus \{0\}$.
 Let $h: [0,1]\times V\setminus \{x^{0}\}\rightarrow \mathbb{R}$ and
  $|h(t,x)|=|h_{t,x}|<C$ where $x_{0}$ is a root of Cayley tree and $C$ is a
constant which does not depend on $t$.  For some $n\in\mathbb{N},$
$\sigma_n:x\in V_n\mapsto \sigma(x)$ and $Z_n$ is the
corresponding partition function we consider the probability
distribution $\mu^{(n)}$ on $\Omega_{V_n}$ defined by
\begin{equation}\label{e2}\mu^{(n)}(\sigma_n)=Z_n^{-1}\exp\left(-\beta H(\sigma_n)
+\sum_{x\in W_n}h_{\sigma(x),x}\right),\end{equation}
\begin{equation}\label{e3}Z_n=\int\!\!\!...\!\!\!\!\!\int\limits_{\Omega^{(p)}_{V_{n-1}}} \exp\left(-\beta
H({\widetilde\sigma}_n) +\sum_{x\in
W_{n}}h_{{\widetilde\sigma}(x),x}\right)
\lambda^{(p)}_{V_{n-1}}({d\widetilde\sigma_n}),\end{equation}
where
$$\underbrace{\Omega_{W_{n}}\times\Omega_{W_{n}}\times...\times\Omega_{W_{n}}}_{3\cdot
2^{p-1}}=\Omega^{(p)}_{W_{n}},\ \ \
\underbrace{\lambda_{W_{n}}\times\lambda_{W_{n}}\times...\times\lambda_{W_{n}}}_{3\cdot
2^{p-1}}=\lambda^{(p)}_{W_{n}}, \ n,p\in \mathbb{N},$$ Let
$\sigma_{n-1}\in\Omega_{V_{n-1}}$ and
$\sigma_{n-1}\vee\omega_n\in\Omega_{V_n}$ is the concatenation of
$\sigma_{n-1}$ and $\omega_n.$ For $n\in \mathbb{N}$ we say that
the probability distributions $\mu^{(n)}$ are compatible if
$\mu^{(n)}$ satisfies the following condition:\vskip 0.1 truecm
\begin{equation}\label{e4}\int\!\!\!\!\!\!\!\!\!\int\limits_{\Omega_{W_n}\times\Omega_{W_n}}
\mu^{(n)}(\sigma_{n-1}\vee\omega_n)(\lambda_{W_n}\times
\lambda_{W_n})(d\omega_n)=
\mu^{(n-1)}(\sigma_{n-1}).\end{equation}\vskip 0.1 truecm

By Kolmogorov'sigma extension theorem there exists a unique
measure $\mu$ on $\Omega_V$ such that, for any $n$ and
$\sigma_n\in\Omega_{V_n}$, $\mu \left(\left\{\sigma
|_{V_n}=\sigma_n\right\}\right)=\mu^{(n)}(\sigma_n)$. The measure
$\mu$ is called {\it splitting Gibbs measure} corresponding to
Hamiltonian (\ref{e1}) and function $x\mapsto h_x$, $x\neq x^0$ (see \cite{ehr2012},\cite{13},\cite{uar}).\\
Denote
\begin{equation}\label{e20}K(t,u,v)=\exp\left\{J_{3}\beta\xi_{1}\left(t,u,v\right)+J\beta\xi_{2}\left(u,v\right)
+J_{1}\beta\left(\xi_{3}\left(t,u\right)+\xi_{3}\left(t,v\right)\right)+\alpha\beta(u+v)\right\},\end{equation}
and
$$f(t,x)=\exp(h_{t,x}-h_{0,x}), \ \ (t,u,v)\in [0,1]^{3},\ x\in
V\setminus\{x^{0}\}.$$\vskip 0.3truecm

The following statement describes conditions on $h_x$ guaranteeing
compatibility of the corresponding distributions
$\mu^{(n)}(\sigma_n).$
 \begin{pro}\label{p1}\cite{b5} The measure
$\mu^{(n)}(\sigma_n)$, $n=1,2,\ldots$ satisfies the consistency
condition (\ref{e4}) iff for any $x\in V\setminus\{x^0\}$ the
following equation holds:
\begin{equation}\label{e5} f(t,x)=\prod_{>y,z<\in S(x)}
\frac{\int_0^1\int_0^1K(t,u,v)f(u,y)f(v,z)dudv}{\int_0^1\int_0^1K(0,u,v)f(u,y)f(v,z)dudv},
\end{equation} where $S(x)=\{y,z\},\ <y,x,z>$ is a ternary
neighbor.\end{pro}

\section{Existence of fixed point of the operator $\mathcal{L}$} Now we prove that there exist at least one
fixed point of Lyapunov integral equation, namely there is a
splitting Gibbs measure corresponding to Hamiltonian (\ref{e1}).

\begin{pro}\label{p2} Let $J_{3}=J=\alpha=0$ and $J_{1}\neq 0$. Then (\ref{e5}) is
equivalent to
\begin{equation}\label{e17} f(t,x)=\prod_{y\in
S(x)}\frac{\int_0^1\exp\left\{J_{1}\beta\xi_{3}(t,u)\right\}
f(u,y)du}{\int_0^1\exp\left\{J_{1}\beta\xi_{3}(0,u)\right\}
f(u,y)du},\end{equation}
 where $f(t,x)=\exp(h_{t,x}-h_{0,x}), \ t\in [0,1],\ x\in V.$
\end{pro}
\begin{proof} For $J_{3}=J=\alpha=0$ and $J_{1}\neq 0$ one get
$K(t,u,v)=\exp\left\{J_{1}\beta\left(\xi_{3}\left(u,t\right)+\xi_{3}\left(v,t\right)\right)\right\}.$
Then (\ref{e5}) can be written as $$f(t,x)=\prod_{>y,z<\in S(x)}
\frac{\int_0^1\int_0^1\exp\left\{J_{1}\beta\left(\xi_{3}\left(t,u\right)+\xi_{3}\left(t,v\right)\right)\right\}
f(u,y)f(v,z)dudv}{\int_0^1\int_0^1\exp\left\{J_{1}\beta\left(\xi_{3}\left(0,u\right)+\xi_{3}\left(0,v\right)\right)\right\}f(u,y)f(v,z)dudv}=$$

\begin{equation}\label{e16}\prod_{>y,z<\in
S(x)}\frac{\int_0^1\exp\left\{J_{1}\beta\xi_{3}(t,u)\right\}
f(u,y)du \cdot \int_0^1\exp\left\{J_{1}\beta\xi_{3}(t,v)\right\}
f(v,z)dv}{\int_0^1\exp\left\{J_{1}\beta\xi_{3}(0,u)\right\}
f(u,y)du \cdot \int_0^1\exp\left\{J_{1}\beta\xi_{3}(0,v)\right\}
f(v,z)dv}.\end{equation}\\
 Since $>y,z<=S(x)$ equation (\ref{e16}) is equivalent to
 (\ref{e17}).\end{proof}

Now we consider the case $J_{3}\neq 0, \ J=J_{1}=\alpha=0$ for the
model (\ref{e1}) in the class of translational-invariant functions
$f(t,x)$ i.e $f(t,x)=f(t),$ for any $x\in V$. For such functions
equation (\ref{e1}) can be written as
\begin{equation}\label{e23}f(t)=\frac{\int_0^1\!\!\int_0^1K(t,u,v)f(u)f(v)dudv}{\int_0^1\!\!
\int_0^1K(0,u,v)f(u)f(v)dudv},
\end{equation}
 where $K(t,u,v)=\exp\left\{J_{3}\beta\xi_{1}\left(t,u,v\right)+J\beta\xi_{2}\left(u,v\right)
+J_{1}\beta\left(\xi_{3}\left(t,u\right)+\xi_{3}\left(t,v\right)\right)+\alpha\beta(u+v)\right\},$
$f(t)>0, \ t,u\in [0,1].$\\
We shall find positive continuous solutions to (\ref{e23}) i.e.
such that $f\in C^+[0,1]=\{f\in C[0,1]: f(x)\geq 0\}$.

 Define a nonlinear operator $H$ on the cone of positive continuous
functions on $[0,1]:$
$$
(Hf)(t)=\frac{\int_0^1\int_0^1K(t,s,u)f(s)f(u)dsdu}{\int_0^1\int_0^1
K(0,s,u)f(s)f(u)dsdu}.$$

 We'll study the existence of positive
fixed points for the nonlinear operator $H$ (i.e., solutions of
the equation (\ref{e23})). Put $ C_0^+[0,1]=C^+[0,1]\setminus
\{\theta\equiv 0\}.$ Then the set $C^+[0,1]$ is the cone of
positive continuous functions on $[0,1].$

We define the Lyapunov integral operator $\mathcal{L}$ on $C[0,1]$
by the equality (see \cite{b3})

$$\mathcal{L}f(t)=\int_0^1K(t,s,u)f(s)f(u)dsdu.$$

Put
$$\mathcal M_0=\left\{f\in C^+[0,1]: f(0)=1\right\}.$$

\begin{lemma}\label{l3} The equation  $Hf=f$ has a nontrivial positive solution iff
the Lyapunov equation $\mathcal{L}g=g$ has a nontrivial positive
solution.\end{lemma}

\begin{proof} At first we shall prove that the equation
\begin{equation}\label{e2.1}
Hf=f, \,\ f\in C^{+}_{0}[0,1]
\end{equation}
 has a positive solution iff the Lyapunov  equation
\begin{equation}\label{e2.2}
\mathcal{L}g=\lambda g, \,\ g\in C^{+}[0,1]
\end{equation}
has a positive solution in $\mathcal M_0$ for some $\lambda>0$.

Let $\lambda_{0}$ be a positive eigenvalue of the Lyapunov
operator $\mathcal{L}.$ Then there exists $f_{0}\in C_0^+[0,1]$
such that $\mathcal{L}f_{0}=\lambda_{0}f_{0}.$ Take $\lambda\in
(0,+\infty)$, $\lambda\ne\lambda_0$. Define the function
$h_0(t)\in C_0^+[0,1]$ by $h_0(t)={\lambda\over \lambda_0}f_0(t),
\ \ t\in [0,1].$ Then $\mathcal{L} h_0=\lambda h_0,$ i.e., the
number $\lambda$ is an eigenvalue of Lyapunov operator
$\mathcal{L}$ corresponding the eigenfunction $h_0(t)$. It's easy
to check that if the number $\lambda_{0}>0$ is an eigenvalue of
the operator $\mathcal{L}$, then an arbitrary positive number is
eigenvalue of the operator $\mathcal{L}$. Now we shall prove the
lemma. Let equation (11) holds then the function
$\frac{1}{\lambda}g(t)$ be a fixed point of the operator
$\mathcal{L}$. Analogously, since $H$ is non-linear operator we
can correspond to the fixed point if there exist any
eigenvector.\end{proof}

\begin{pro}\label{p4} The equation
\begin{equation}\label{e2.3}
 \mathcal{L}f=\lambda f, \,\,\ \lambda>0
\end{equation}
 has at least one solution in $C_0^+[0,1].$
\end{pro}

\begin{proof} Clearly, that the Lyapunov operator $\mathcal{L}$ is a
compact on the cone $C^+[0,1]$. By the other hand we have

$$\mathcal{L}f(t)\geq m \left(\int_0^1f(s)ds\right)^{2},$$
for all $f\in C^+[0,1]$, where $m = \min K(t,s,u)>0.$

Put $\Gamma=\{f: \,\ \|f\|=r, \,\ f\in C[0,1] \}$. We define the
set $\Gamma_+ $ by

$$\Gamma_+=\Gamma\cap C^+[0,1].$$

Then we obtain

$$\inf_{f\in \Gamma_+}\|\mathcal{L}f\|>0.$$

Then by Schauder's theorem (see \cite{b4}, p.20) there exists a
number $\lambda_0>0$ and a function $f_0\in \Gamma_+$ such that,
$\mathcal{L}f_0=\lambda_0 f_0.$\end{proof}

Denote by $N_{fix.p}(H)$ and $N_{fix.p}(\mathcal{L})$ are the set
of positive numbers of nontrivial positive fixed points of the
operators $N_{fix.p}(H)$ and $N_{fix.p}(\mathcal{L})$,
respectively. By Lemma \ref{l3} and Proposition \ref{p4} we can
conclude that:

\begin{pro}\label{p5} a)The equation (\ref{e2.1}) has at least one solution in
$C_0^+[0,1]$.

 b) The equality
$N_{fix.p}(H)=N_{fix.p}(\mathcal{L})$ is hold.\end{pro}

From Proposition \ref{p1} and Proposition \ref{p5} we get the
following theorem.

\begin{thm}\label{th5} The set of {\it splitting Gibbs measures} corresponding to
Hamiltonian (\ref{e1}) is non-empty. \end{thm}

\section{The uniqueness of fixed point of the operator $\mathcal{L}$}

In this section we shall give a condition of the uniqueness of
fixed point of the operator $\mathcal{L}$.

\begin{thm}\label{t3}
Let the kernel $K(t,u,v)$ satisfies the condition
\begin{equation}\label{I}
\max_{(t,u,v)\in [0,1]^{3}} K(t,u,v)<c\min_{(t,u,v)\in [0,1]^{3}}
K(t,u,v), \ \ c\in \left(1, \frac{1}{2}\sqrt{\sqrt{17}+1}\right).
\end{equation}
Then the operator $\mathcal{L}$ has the unique fixed point in
$C_0^+[0,1]$.
\end{thm}
\begin{proof}
Let $\max_{(t,u,v)\in [0,1]^{3}} K(t,u,v)=\Omega$ and
$\min_{(t,u,v)\in [0,1]^{3}} K(t,u,v)=\omega$. At first we shall
prove that if $g\in C_0^+[0,1]$ is a solution of the equation
$\mathcal{L}f=f$ then $g\in \mathcal G$ where $$\mathcal
G=\left\{f\in C[0,1]: {\omega\over \Omega^{2}} \leq f(t)\leq
{\Omega\over \omega^{2}}\right\}.$$
 Let $s\in \mathcal L(C^+[0,1])$ be an arbitrary
function. Then there exists a function $h\in C^+[0,1]$ such that
$s=\mathcal{L}h$. Since $s$ is continuous on $[0,1]$, there exists
$t_1,t_2\in [0,1]$ such that
$$s_{\min}=\min_{t\in[0,1]}s(t)=s(t_1)=(\mathcal{L}h)(t_1), \ \ s_{\max}=\max_{t\in[0,1]}s(t)=s(t_2)=(\mathcal{L}h)(t_2).$$
Consequently we get
\begin{equation}\label{f1}s_{\min}\geq
\omega\int^1_0\int^1_0h(u)h(v)dudv\geq
\omega\int^1_0\int^1_0{K(t_2,u,v)\over
\Omega}h(u)h(v)dudv={\omega\over \Omega}s_{max}.\end{equation}
Since $g$ is a fixed point of the operator $\mathcal L$ we have
$\|g\|\leq \Omega\|g\|^2 \Rightarrow \|g\|\geq
\frac{1}{\Omega}$.\\
From (\ref{f1})
$$g(t)\geq g_{\min}=\min_{t\in [0,1]}g(t)\geq {\omega\over
\Omega}\|g\| \ \Rightarrow g(t)\geq {\omega\over \Omega^{2}}.$$
Similarly,
$$g(t)=(\mathcal Lg)(t)\geq \omega\int^1_0\int^1_0g(u)g(v)dudv\geq \omega g_{\min}^2 \ \Rightarrow g_{\min}\leq \frac{1}{\omega}.$$ Hence
$$g(t)\leq g_{\max}\leq {\Omega\over \omega}g_{\min}\leq {\Omega\over \omega^{2}}.$$ Thus we have
$g\in \mathcal G$.

Now we show that $\mathcal{L}$ has the unique fixed point. By
Theorem  $\mathcal{L}g=g$ has at least one solution.
 Assume that there are two solutions
$g_1\in C_0^+[0,1]$ and $g_2\in C_0^+[0,1]$, i.e $\mathcal
Lg_i=g_i$, $i=1,2$. Put $\xi(t)=g_1(t)-g_2(t)$. Then $\xi(t)$
changes its sign on $[0,1]$ and we get
$$\max_{t\in[0,1]}\left|\xi(t)-\left(\frac{\omega^{2}}{\Omega^2}+\frac{\Omega^{2}}{\omega^2}\right)\int^1_0\xi(s)ds\right|\geq{1\over 2}\|\xi\|,$$
$$\xi(t)=2\int^1_0\int^1_0K(t,u,v)\left(g_{1}(u)g_{1}(v)-g_{2}(u)g_{2}(v)\right)dudv.$$
The last equation can be written as
$$\xi(t)=\int^1_0\int^1_0K(t,u,v)\eta(u,v)\left(|\xi(u)-\xi(v)|+\xi(u)+\xi(v)\right)dudv,$$
where
$$\min\{g_1(t), g_2(t)\}\leq \eta(t)\leq \max\{g_1(t), g_2(t)\}, \, t\in [0,1].$$
Since $g_{i}(t)\in\mathcal G , \ i=\overline{1,2}$ we get
$\frac{\omega}{\Omega^{2}}\leq \eta(u,v)\leq
\frac{\Omega}{\omega^{2}}, \, (u,v)\in [0,1]^{2}.$ Hence
$$\left|2\cdot K(t,u,v)\eta(u,v)-\left(\frac{\Omega^2}{\omega^2}+\frac{\omega^2}{\Omega^2}\right)\right|\leq
\frac{\Omega^2}{\omega^2}-\frac{\omega^2}{\Omega^2}.$$ Then
\begin{equation}\label{ee6}
\left|\xi(t)-\left(\frac{\Omega^2}{\omega^2}+\frac{\omega^2}{\Omega^2}\right)\int^1_0\int^1_0
\left(|\xi(u)-\xi(v)|+\xi(u)+\xi(v)\right)dudv\right|\leq
\left(\frac{\Omega^2}{\omega^2}-\frac{\omega^2}{\Omega^2}\right)\|\xi\|.
\end{equation}
Assume the kernel $K(t,u,v)$ satisfies the condition (\ref{I}).
Then $\Omega^{4}-\omega^{4}<(\Omega\omega)^{2} \Rightarrow
\Omega<c \omega$ but it's contradict to the following: if  $\xi\in
C[0,1]$ changes its sign on $[0,1]$ then for every $a\in \R$ the
following inequality holds $\|\xi-a\|\geq \frac{1}{2}\|\xi\|.$
This completes the proof.
\end{proof}

\begin{thm}\label{th6}
Let $k\geq 2$. If the function $K(t,u,v)$ which defined in
(\ref{e20}) satisfies the condition (\ref{I}), then the model
(\ref{e1}) has the unique translational invariant Gibbs measure.
\end{thm}

\end{document}